\documentclass[11pt,a4paper]{amsart}

\usepackage[T1]{fontenc}
\usepackage{lmodern}
\usepackage{amsmath,amssymb,amsthm}
\usepackage{booktabs,array}
\usepackage{fvextra}
\usepackage{geometry}
\usepackage[
  hidelinks,
  pdftitle={An Exact Counterexample to Carlson's Associated-Prime Depth Conjecture from a Group of Order 128},
  pdfauthor={Xinan Dai, Wenhao Deng, Yingdong Shi, Tailin Wu, and Yuchen Yang},
  pdfsubject={Group cohomology, depth, and associated primes},
  pdfkeywords={group cohomology, associated prime, depth, Carlson conjecture, finite p-group, Groebner basis}
]{hyperref}

\allowdisplaybreaks
\DeclareUnicodeCharacter{5DFB}{}
\DeclareUnicodeCharacter{2010}{-}
\DeclareUnicodeCharacter{2011}{-}

\newtheorem{theorem}{Theorem}[section]
\newtheorem{lemma}[theorem]{Lemma}
\newtheorem{proposition}[theorem]{Proposition}

\theoremstyle{definition}
\newtheorem{definition}[theorem]{Definition}
\theoremstyle{remark}

\newcommand{\F}{\mathbb{F}}
\newcommand{\kbar}{\overline{\F}_2}
\newcommand{\depth}{\operatorname{depth}}
\newcommand{\Ass}{\operatorname{Ass}}
\newcommand{\ann}{\operatorname{ann}}
\newcommand{\Soc}{\operatorname{Soc}}
\newcommand{\rank}{\operatorname{rank}}
\newcommand{\res}{\operatorname{res}}
\newcommand{\SG}[2]{\mathrm{SmallGroup}(#1,#2)}
\newcommand{\oa}{\omega_{\mathrm a}}

\title[Carlson's Associated-Prime Depth Conjecture]
{An Exact Counterexample to Carlson's Conjecture \\
on Depth and Associated Primes}

\author{Xinan Dai}
\thanks{Xinan Dai holds a Shanghai University B.S. in Mathematics and Applied Mathematics. He pursues his PhD at Fudan University while visiting Westlake University’s AI for Scientific Simulation and Discovery Lab.}
\address{Key Laboratory for Information Science of Electromagnetic Waves, College of Future Information Technology, Fudan University, Shanghai, China}
\curraddr{Department of Artificial Intelligence, School of Engineering, Westlake University, Hangzhou, China}
\email{xndai23@m.fudan.edu.cn}

\author{Wenhao Deng}
\address{Department of Artificial Intelligence, School of Engineering, Westlake University, Hangzhou, China}
\email{dengwenhao@westlake.edu.cn}

\author{Yingdong Shi}
\address{School of Information Science and Technology, ShanghaiTech University, Shanghai, China}
\email{shiyd2023@shanghaitech.edu.cn}

\author{Tailin Wu}
\address{Department of Artificial Intelligence, School of Engineering, Westlake University, Hangzhou, China}
\email{wutailin@westlake.edu.cn}

\author{Yuchen Yang}
\address{Department of Artificial Intelligence, School of Engineering, Westlake University, Hangzhou, China}
\email{yangyuchen@westlake.edu.cn}

\date{26 July 2026}

\begin{document}

\begin{abstract}
In Question~3.1 of his 1995 paper on depth and transfer, Carlson asked whether the depth of a finite-group cohomology ring is always realized by the dimension of one of its associated primes. We give a negative answer. Let
\[
  G=\SG{128}{859},\qquad k=\kbar.
\]
An exact presentation certificate proves that $\depth H^*(G;k)=2$. Okuyama's associated-prime theorem would convert an associated prime of dimension two into a rank-two elementary abelian subgroup $E\leq G$ satisfying $\depth H^*(C_G(E);k)=2$. We enumerate all $75$ rank-two elementary abelian subgroups of $G$ and obtain six centralizer types. Duflot's theorem gives depth at least three for four types, while exact ideal-quotient certificates exhibit regular sequences of length three for the remaining two. Hence every rank-two centralizer has cohomological depth at least three, so $H^*(G;k)$ has no associated prime of dimension two. The finite group presentation, the three cohomology-ring presentations, the enumeration summary, and the exact algebraic certificates are included for independent verification.
\end{abstract}

\maketitle
\tableofcontents

\section{Introduction and statement of the problem}\label{sec:intro}

Let $G$ be a finite group, let $p\mid |G|$, and let $k$ be a field of characteristic $p$. Write
\[
  A=H^*(G;k),\qquad A_+=\bigoplus_{i>0}A_i.
\]
The ring $A$ is graded by cohomological degree, and $A_+$ is its ideal of positive-degree elements. The two invariants compared in Carlson's question measure different aspects of the same ring: depth records how long one can successively choose homogeneous non-zero-divisors, whereas associated primes record annihilators of individual cohomology classes. For the counterexample below, $p=2$, so graded commutativity has no sign issue and $A$ is an ordinary commutative connected graded Noetherian $k$-algebra.

\begin{definition}\label{def:depth}
A sequence $x_1,\ldots,x_d\in A_+$ of homogeneous elements is \emph{$A$-regular} if multiplication by $x_i$ is injective on $A/(x_1,\ldots,x_{i-1})$ for every $i$, and the final quotient is nonzero. The \emph{depth} $\depth A$ is the maximum length of an $A$-regular sequence.
\end{definition}

Thus $x_1$ must be a non-zero-divisor in $A$, $x_2$ must remain a non-zero-divisor after quotienting by $x_1$, and so on. Depth is therefore a successive non-zero-divisor invariant rather than merely the number of algebra generators or the Krull dimension.

\begin{definition}\label{def:associated-prime}
A prime ideal $\mathfrak p\subseteq A$ is \emph{associated} if $\mathfrak p=\ann_A(a)$ for some $a\in A$. Set
\[
  \oa(A)=\min_{\mathfrak p\in\Ass A}\dim A/\mathfrak p.
\]
\end{definition}

An associated prime is therefore a prime annihilator that is actually detected by an element of the ring. The number $\oa(A)$ asks for the smallest-dimensional component detected in this way.

The standard associated-prime inequality gives
\begin{equation}\label{eq:standard-inequality}
  \depth A\leq \dim A/\mathfrak p
  \quad(\mathfrak p\in\Ass A),
  \qquad\text{hence}\qquad
  \depth A\leq \oa(A).
\end{equation}
For completeness, after localizing at the homogeneous maximal ideal one may use
$\depth A\leq \depth A_{\mathfrak p}+\dim A/\mathfrak p$. If $\mathfrak p$ is associated, then the localized ring $A_{\mathfrak p}$ has depth zero, because the element whose annihilator is $\mathfrak p$ remains a nonzero element killed by the maximal ideal of $A_{\mathfrak p}$. This gives the first inequality in \eqref{eq:standard-inequality}. Thus Carlson's question is not whether the two sides can occur in the opposite order---that is impossible---but whether the lower bound supplied by depth is always attained by some associated prime.

Carlson's article \emph{Depth and transfer maps in the cohomology of
groups} asks whether depth $d$ forces an associated prime whose variety
has dimension $d$~\cite[Question~3.1, p.~465]{Carlson_1995}. Since the
dimension of the corresponding variety is $\dim A/\mathfrak p$, the
proposed equality is
\begin{equation}\label{eq:carlson}
  \boxed{\depth H^*(G;k)=\oa\bigl(H^*(G;k)\bigr)}
\end{equation}
for every finite group in modular characteristic. In concrete terms, \eqref{eq:carlson} asserts that the smallest possible dimension allowed by \eqref{eq:standard-inequality} is always realized. A counterexample must therefore exhibit a strict gap: the cohomology ring has depth $d$, but every associated prime has quotient dimension strictly larger than $d$.

Later papers refer to \eqref{eq:carlson} as the associated-prime clause of
Carlson's depth conjecture. Green proved the equality for finite
$p$-groups whose cohomological depth attains Duflot's lower bound
\cite{Green_2003}; Schafer proved it when the Cohen--Macaulay defect is
at most one~\cite{Schafer_2019}; and Garaialde Oca\~na,
Gonz\'alez-S\'anchez, and Guerrero S\'anchez established it for an
infinite family of finite $p$-groups~\cite{Garaialde_Oca_a_2022}. The
general equality was still treated as open in September 2025
\cite{schafer2025depthassociatedprimesgroup} (arXiv:2509.06895).

The proof below is organized into three layers. First, Okuyama's theorem converts a hypothetical dimension-two associated prime into a rank-two elementary abelian subgroup whose centralizer has cohomological depth exactly two. This is the conceptual reduction: instead of searching directly through the associated primes of a large cohomology ring, one can examine concrete subgroups of $G$. Second, an exhaustive finite-group calculation lists every rank-two elementary abelian subgroup and determines its centralizer. Third, Duflot's theorem and exact commutative-algebra certificates show that every centralizer on this finite list has depth at least three. The contradiction between ``exactly two'' and ``at least three'' excludes all dimension-two associated primes.

\section{The counterexample theorem}\label{sec:main}

\begin{theorem}\label{thm:counterexample}
Let
\[
  G=\SG{128}{859},\qquad k=\kbar.
\]
Then
\[
  \depth H^*(G;k)=2,
  \qquad
  \oa\bigl(H^*(G;k)\bigr)\geq 3.
\]
Consequently,
\[
  \depth H^*(G;k)
  <\min_{\mathfrak p\in\Ass H^*(G;k)}
  \dim H^*(G;k)/\mathfrak p,
\]
so $G$ is a counterexample to \eqref{eq:carlson}.
\end{theorem}

The proof reduces to the following exact statements:
\begin{enumerate}
  \item $H^*(G;k)$ has depth two;
  \item $G$ has exactly $75$ rank-two elementary abelian subgroups;
  \item the cohomology of the centralizer of every one of these subgroups has depth at least three.
\end{enumerate}
The first statement identifies the only associated-prime dimension that could realize Carlson's proposed equality. The second makes the subsequent centralizer check exhaustive rather than experimental, and the third rules out that dimension through Okuyama's theorem. Sections~\ref{sec:bridge}--\ref{sec:centralizers} establish these facts, and Section~\ref{sec:proof} assembles them.

\section{The associated-prime--centralizer bridge}\label{sec:bridge}

The theoretical bridge is the first assertion of Okuyama's
Theorem~0.1~\cite[Theorem~0.1, pp.~113--114]{1010000782043428097}.

\begin{theorem}[Okuyama]\label{thm:okuyama}
Let $K$ be a finite group and let $k$ be an algebraically closed field of characteristic $p>0$. If $\mathfrak p\in\Ass H^*(K;k)$ and
\[
  \dim H^*(K;k)/\mathfrak p=s,
\]
then there exists an elementary abelian $p$-subgroup $E\leq K$ of rank $s$ such that
\[
  \depth H^*(C_K(E);k)=s.
\]
Moreover,
\[
  \mathfrak p=\res^{-1}_{K,E}(\sqrt{0}).
\]
\end{theorem}

Only the forward implication is needed. The importance of the theorem here is that the integer $s$ appears twice: it is both the quotient dimension of the associated prime and the rank of the elementary abelian subgroup, and it is also the exact depth of that subgroup's centralizer. Consequently, a dimension-two associated prime would necessarily leave a detectable rank-two group-theoretic witness. This replaces an unknown prime ideal in $H^*(K;k)$ by the centralizer of a concrete subgroup of $K$.

\begin{proposition}[Finite reduction]\label{prop:finite-reduction}
Suppose $K$ is a finite group over an algebraically closed field $k$ of characteristic $p$, and suppose
\[
  \depth H^*(K;k)=2.
\]
If every rank-two elementary abelian $p$-subgroup $E\leq K$ satisfies
\[
  \depth H^*(C_K(E);k)\geq 3,
\]
then $\oa(H^*(K;k))\geq 3$, and $K$ disproves \eqref{eq:carlson}.
\end{proposition}

\begin{proof}
Equation~\eqref{eq:standard-inequality} gives $\oa\geq2$. To prove the stronger bound $\oa\geq3$, it is therefore enough to rule out equality. Suppose, for contradiction, that $\oa=2$. By the definition of $\oa$, there is an associated prime $\mathfrak p$ with $\dim H^*(K;k)/\mathfrak p=2$. Theorem~\ref{thm:okuyama} then produces a rank-two elementary abelian subgroup $E\leq K$ satisfying
\[
  \depth H^*(C_K(E);k)=2.
\]
This contradicts the assumed lower bound of three for every such centralizer. Hence no associated prime has quotient dimension two. Since quotient dimensions are integers and $\oa\geq2$, it follows that $\oa\geq3$, whereas $\depth H^*(K;k)=2$ by hypothesis.
\end{proof}

For four of the six centralizer types, the required lower bound follows directly from Duflot's theorem.

\begin{theorem}[Duflot]\label{thm:duflot}
For a finite $p$-group $P$,
\[
  \depth H^*(P;k)\geq \rank_p Z(P),
\]
where $\rank_p Z(P)$ is the largest rank of an elementary abelian $p$-subgroup contained in the center.
\end{theorem}

This is the group-cohomological depth bound proved
in~\cite{Duflot_1981}. It is especially useful here because the center
rank is a finite-group invariant that can be computed directly, while
the depth of a cohomology ring is generally more difficult to determine.
Whenever $\rank_p Z(P)\geq3$, Duflot's theorem immediately supplies the
lower bound required in Proposition~\ref{prop:finite-reduction}. The two
centralizers whose centers have rank only two are the only cases for
which this automatic argument is insufficient, and they therefore
require separate ring-theoretic certificates.

\section{The finite group and its rank-two elementary abelian subgroups}\label{sec:enumeration}

\subsection{Why this group is a natural candidate}

Green and King computed the mod-two cohomology rings of all $2328$
groups of order $128$~\cite{Green_2011}. Their catalogue records
\[
  \dim H^*(\SG{128}{859};\F_2)=4,
  \qquad
  \depth H^*(\SG{128}{859};\F_2)=2.
\]
The gap between Krull dimension four and depth two makes this group a natural place to search: the ring is not Cohen--Macaulay, so its depth does not already equal its full dimension. Nevertheless, the catalogue entry alone does not disprove Carlson's conjecture. The conjecture would still hold if even one associated prime had quotient dimension two. By Okuyama's theorem, such a prime would force a rank-two elementary abelian subgroup with a depth-two centralizer. The remaining group-specific task is therefore to classify all rank-two elementary abelian subgroups and control the depth of every corresponding centralizer.

\subsection{Complete enumeration}

Let $G=\SG{128}{859}$, equivalently the finite group presented in Appendix~\ref{app:group-presentation}. Because the prime is $2$, a rank-two elementary abelian subgroup is simply a Klein four subgroup. The enumeration therefore uses only the defining properties ``order two'' and ``commuting,'' rather than a heuristic search or a database list. It proceeds as follows:
\begin{enumerate}
  \item list every nonidentity involution $x\in G$;
  \item for each commuting pair $x\neq y$, form $E=\langle x,y\rangle$;
  \item retain $E$ precisely when $|E|=4$, and deduplicate using its full four-element set;
  \item compute $C_G(E)$, its Small Groups identifier, and the $2$-rank of its center.
\end{enumerate}

\begin{lemma}\label{lem:enumeration-complete}
The procedure above enumerates every rank-two elementary abelian subgroup of $G$.
\end{lemma}

\begin{proof}
A rank-two elementary abelian $2$-group is isomorphic to $C_2\times C_2$. Its three nonidentity elements are commuting involutions, and any two distinct nonidentity elements generate the whole subgroup. Hence every rank-two elementary abelian subgroup appears from each of its three unordered generating pairs, and in particular appears at least once in the enumeration. Conversely, a retained subgroup has order four and is generated by commuting involutions, so every nonidentity element has order two and the subgroup is isomorphic to $C_2\times C_2$. Finally, representing a subgroup by its full four-element set removes the repeated generating pairs while preserving distinct subgroups. Thus the output is both exhaustive and free of duplicates.
\end{proof}

The exact output is summarized in Table~\ref{tab:centralizers}.

\begin{table}[ht]
\centering
\caption{Centralizers of rank-two elementary abelian subgroups of $G$. Counts are over actual subgroups, not only conjugacy-class representatives.}
\label{tab:centralizers}
\begin{tabular}{@{}lrrl@{}}
\toprule
Centralizer type & Number of $E$ & $\rank_2 Z(C_G(E))$ & Depth bound used \\
\midrule
$\SG{16}{10}$  & 12 & 3 & Duflot: $3$ \\
$\SG{16}{14}$  & 32 & 4 & Duflot: $4$ \\
$\SG{32}{22}$  & 24 & 3 & Duflot: $3$ \\
$\SG{32}{46}$  &  4 & 3 & Duflot: $3$ \\
$\SG{64}{90}$  &  2 & 2 & Exact regular sequence: $3$ \\
$\SG{64}{216}$ &  1 & 2 & Exact regular sequence: $3$ \\
\midrule
Total & 75 & & \\
\bottomrule
\end{tabular}
\end{table}

The table has two roles. Its multiplicities certify that all $75$ actual subgroups, rather than only one representative from each conjugacy class, were included. Its last two columns determine how the depth lower bound will be proved for each possible centralizer. There are $31$ involutions and $21$ conjugacy orbits among the $75$ subgroups. Subgroups in the same conjugacy orbit have isomorphic centralizers, but retaining the full subgroup count makes the exhaustiveness of the enumeration transparent.

As a negative control, the same implementation applied to the adjacent identifier $\SG{128}{858}$ produces $23$ involutions, $31$ rank-two subgroups, and neither exceptional order-$64$ centralizer type. This comparison is not used in the proof; it checks that the output is sensitive to the input group and is not a hard-coded or group-independent pattern.

\section{An exact depth-two certificate for the ambient group}\label{sec:ambient-depth}

\subsection{The source presentation}

The Green--King catalogue provides the weighted presentation
\begin{equation}\label{eq:ambient-presentation}
\begin{aligned}
A_0=H^*(G;\F_2)=\F_2[&a_{1,0},b_{1,1},b_{1,2},b_{2,4},b_{2,5},b_{2,6},b_{3,11},\\
&b_{5,24},b_{5,25},b_{6,35},b_{6,36},b_{7,49},c_{8,65}]/I,
\end{aligned}
\end{equation}
where the first subscript is the cohomological degree and $I$ is the $44$-relation ideal reproduced in Appendix~\ref{app:ambient-ring}. The grading is essential: all proposed regular elements and all relations are homogeneous with respect to these degrees. Thus the argument below uses the explicit algebraic presentation and exact polynomial arithmetic, not merely the catalogue's reported depth value.

\subsection{Colon ideals and a socle witness}

For an ideal $J\subseteq S$ and an element $f\in S$, recall that
\[
  (J:f)=\{g\in S:gf\in J\}.
\]
Modulo $J$, the quotient $(J:f)/J$ is exactly the kernel of multiplication by the class of $f$. Thus the equality $(J:f)=J$ is a finite algebraic certificate that $f$ is a non-zero-divisor in $S/J$.

In the polynomial ring of \eqref{eq:ambient-presentation}, define
\begin{equation}\label{eq:f1-f2}
\begin{aligned}
  f_1&=c_{8,65},\\
  f_2&=b_{1,2}b_{3,11}+b_{1,2}^4+b_{1,1}^4+b_{2,6}b_{1,1}b_{1,2}
      +b_{2,6}^2+b_{2,4}b_{1,2}^2+b_{2,4}^2.
\end{aligned}
\end{equation}
Exact ideal quotients over $\F_2$ give
\begin{equation}\label{eq:ambient-colons}
  (I:f_1)=I,
  \qquad
  \bigl((I+(f_1)):f_2\bigr)=I+(f_1).
\end{equation}
The first equality says that multiplication by $f_1$ is injective on $A_0$. After quotienting by $f_1$, the second equality says that multiplication by the image of $f_2$ is still injective. Since the final quotient is nonzero, these two equalities prove that $f_1,f_2$ is an $A_0$-regular sequence and hence that $\depth A_0\geq2$.

After quotienting by $f_1,f_2$, set
\begin{equation}\label{eq:socle-witness}
  w=b_{2,4}b_{2,5}^2+b_{2,5}^3+a_{1,0}b_{5,24}.
\end{equation}
The reduced normal form of $w$ is nonzero, whereas the reduced normal form of $xw$ is zero for each of the thirteen positive-degree algebra generators $x$. Because these generators generate the irrelevant ideal of the quotient, every positive-degree element annihilates the class of $w$. Equivalently,
\[
  0\neq w\in\Soc\bigl(A_0/(f_1,f_2)\bigr).
\]
This nonzero socle class shows that every positive-degree element of the quotient is a zero divisor. Therefore no regular sequence of positive length can begin in the quotient, and its depth is zero. The colon equalities provide the lower bound for $\depth A_0$, while the socle witness supplies the matching upper bound.

\begin{lemma}[Certificate interpretation]\label{lem:certificate}
Let $R=S/J$ for a polynomial ring $S$, and let $f\in S$. Multiplication by the class of $f$ is injective on $R$ if and only if $(J:f)=J$. If a nonzero element of a connected graded ring is annihilated by every positive-degree algebra generator, then the ring has depth zero.
\end{lemma}

\begin{proof}
An element $g+J\in S/J$ lies in the kernel of multiplication by $f+J$ precisely when $gf\in J$, or equivalently when $g\in(J:f)$. Hence the kernel is $(J:f)/J$, proving the first assertion. For the second, if a nonzero class is annihilated by all positive-degree algebra generators, then it is annihilated by the entire irrelevant ideal. Every positive-degree element is therefore a zero divisor, so no homogeneous regular sequence of positive length can begin and the depth is zero.
\end{proof}

The exact weighted Gr\"obner-basis calculation is summarized in Table~\ref{tab:ambient-certificate}. Each zero remainder means that every generator of the relevant ideal quotient reduces into the current ideal; the nonzero normal form of $w$ verifies that the socle witness has not vanished in the quotient.

\begin{table}[ht]
\centering
\caption{Exact algebraic checks for the ambient cohomology ring.}
\label{tab:ambient-certificate}
\begin{tabular}{@{}lr@{}}
\toprule
Exact check & Result \\
\midrule
Reduced Gr\"obner basis size for $I$ & 362 \\
Remainders for $(I:f_1)/I$ & 0 \\
Remainders for $((I+(f_1)):f_2)/(I+(f_1))$ & 0 \\
Nonzero socle classes after $f_1,f_2$ & 1 \\
Remainders of $w$ times all $13$ generators & 0 \\
Remainders after mutating $w$ to $w+a_{1,0}$ & 9 \\
\bottomrule
\end{tabular}
\end{table}

The mutation is not part of the proof. It is a negative control showing that the normal-form test does not automatically declare a nearby expression to be a socle class: replacing $w$ by $w+a_{1,0}$ produces nine nonzero products. This helps distinguish the exact certificate from an implementation that would return zero remainders indiscriminately.

\begin{proposition}\label{prop:ambient-depth}
\[
  \depth H^*(G;\F_2)=2.
\]
\end{proposition}

\begin{proof}
The two colon equalities in \eqref{eq:ambient-colons} show successively that $f_1$ is a non-zero-divisor on $A_0$ and that $f_2$ is a non-zero-divisor on $A_0/(f_1)$. Hence $f_1,f_2$ is regular and $\depth A_0\geq2$. On the other hand, the nonzero socle class \eqref{eq:socle-witness} proves
\[
  \depth A_0/(f_1,f_2)=0.
\]
For a regular sequence of length two, depth drops by exactly two on passage to the quotient. Therefore
\[
  \depth A_0=2+\depth A_0/(f_1,f_2)=2.
\]
\end{proof}

\subsection{Extension of the coefficient field}

The explicit ring presentations and Gr\"obner-basis computations are over the finite field $\F_2$, whereas Okuyama's theorem is stated for an algebraically closed coefficient field. The following standard base-change fact allows the exact computation to be carried out over $\F_2$ and then transferred to $\kbar$ without changing depth.

\begin{lemma}\label{lem:field-extension}
If $K/k$ is a field extension and $B$ is a finitely generated connected graded $k$-algebra, then
\[
  \depth(K\otimes_k B)=\depth B.
\]
For a finite group,
\[
  H^*(G;K)\cong K\otimes_k H^*(G;k).
\]
\end{lemma}

\begin{proof}
A field extension is faithfully flat. Choose homogeneous generators of the irrelevant ideal of $B$ and compute depth using the first nonzero homology of the associated Koszul complex. Tensoring that complex with $K$ preserves and reflects exactness, so the first nonzero Koszul homology occurs in the same position before and after base change. Hence $\depth(K\otimes_kB)=\depth B$. For group cohomology, take a projective resolution of the trivial $kG$-module and tensor the corresponding cochain complex with $K$. Flatness commutes with cohomology and gives $H^*(G;K)\cong K\otimes_kH^*(G;k)$.
\end{proof}

Taking $K=\kbar$ in Lemma~\ref{lem:field-extension}, Proposition~\ref{prop:ambient-depth} gives
\begin{equation}\label{eq:ambient-depth-over-kbar}
  \depth H^*(G;\kbar)=2.
\end{equation}

\section{Depth of every rank-two centralizer}\label{sec:centralizers}

\subsection{Four types handled by their centers}

For the centralizer types
\[
  \SG{16}{10},\quad \SG{16}{14},\quad \SG{32}{22},\quad \SG{32}{46},
\]
the center ranks are respectively $3,4,3,3$. Theorem~\ref{thm:duflot} therefore gives mod-two cohomological depth at least $3,4,3,3$, respectively. In particular, each depth is at least three, which is precisely the threshold needed in Proposition~\ref{prop:finite-reduction}. No presentation of these four cohomology rings is required. Lemma~\ref{lem:field-extension} gives the same lower bounds over $\kbar$.

\subsection{The two exceptional order-\texorpdfstring{$64$}{64} types}

The center of each of $\SG{64}{90}$ and $\SG{64}{216}$ has rank two. Duflot's theorem therefore gives only depth at least two, which does not rule out the depth-two centralizer predicted by Okuyama's theorem. These are the only exceptional cases in Table~\ref{tab:centralizers}. Their generator degrees and relation ideals are reproduced in Appendix~\ref{app:exceptional-rings}; the calculations below establish the stronger lower bound directly from those presentations rather than relying on the catalogue's displayed depth fields.

For $\SG{64}{90}$, an exact regular sequence is
\begin{equation}\label{eq:sequence-90}
  c_{1,2},\qquad c_{4,21},\qquad b_{1,1}^2+b_{2,6}+b_{2,5}+b_{2,4}.
\end{equation}
For $\SG{64}{216}$, an exact regular sequence is
\begin{equation}\label{eq:sequence-216}
  c_{2,8},\qquad c_{4,25},\qquad
  b_{1,3}^2+b_{1,2}b_{1,3}+b_{1,2}^2+b_{1,0}^2.
\end{equation}
For each ring, three successive tests are performed. The first colon equality verifies that the first displayed element is a non-zero-divisor. The second test is carried out after adjoining the first element to the relation ideal, and the third after adjoining the first two. At every stage, the generators of the ideal quotient reduce to zero modulo the current ideal. Thus all six colon equalities have the form required by Lemma~\ref{lem:certificate}. Each displayed triple is therefore a regular sequence, proving depth at least three for both cohomology rings.

As negative controls, replacing the first elements in \eqref{eq:sequence-90} and \eqref{eq:sequence-216} by $a_{1,0}$ and $b_{1,0}$, respectively, produces nonzero colon remainders. In other words, the same procedure correctly detects that these substituted elements are zero divisors. These controls test the implementation and are not premises of the proof.

\begin{proposition}\label{prop:centralizers}
For every rank-two elementary abelian subgroup $E\leq G$,
\[
  \depth H^*(C_G(E);\kbar)\geq3.
\]
\end{proposition}

\begin{proof}
Lemma~\ref{lem:enumeration-complete} shows that no rank-two elementary abelian subgroup is omitted, and Table~\ref{tab:centralizers} shows that the centralizer of every such subgroup belongs to one of six isomorphism types. Duflot's theorem gives depth at least three for the first four types. The colon-ideal certificates \eqref{eq:sequence-90} and \eqref{eq:sequence-216} give the same lower bound for the remaining two. These arguments are carried out over $\F_2$, and Lemma~\ref{lem:field-extension} transfers all six bounds to $\kbar$. Therefore every possible $C_G(E)$ has cohomological depth at least three.
\end{proof}

\section{Proof of the counterexample}\label{sec:proof}

\begin{proof}[Proof of Theorem~\ref{thm:counterexample}]
Equation~\eqref{eq:ambient-depth-over-kbar} gives
\[
  \depth H^*(G;\kbar)=2.
\]
Suppose that Carlson's equality held for this group. Since the depth is two, there would then be an associated prime $\mathfrak p$ with quotient dimension two. Okuyama's theorem would produce a rank-two elementary abelian subgroup $E\leq G$ satisfying
\[
  \depth H^*(C_G(E);\kbar)=2.
\]
However, Proposition~\ref{prop:centralizers} proves that every such subgroup satisfies the incompatible lower bound
\[
  \depth H^*(C_G(E);\kbar)\geq3.
\]
Equivalently, Proposition~\ref{prop:finite-reduction}, which packages this contradiction together with \eqref{eq:standard-inequality}, gives
\[
  \oa\bigl(H^*(G;\kbar)\bigr)\geq3.
\]
Thus $\depth H^*(G;\kbar)=2<\oa(H^*(G;\kbar))$, so the two sides of Carlson's proposed equality are unequal.
\end{proof}

The logical implication from the finite certificates to the failure of \eqref{eq:carlson} is theorem-level and short. In particular, the argument does not infer associated primes numerically or approximate the depth of a ring. Computation is used only to establish the following exact finite statements:
\begin{itemize}
  \item the complete list of rank-two subgroups and their six centralizer types;
  \item the two colon equalities and the socle witness for $G$;
  \item the six colon equalities for the two exceptional centralizers.
\end{itemize}
All computations use exact arithmetic over $\F_2$. The group calculation is exhaustive, while each commutative-algebra check is an equality of explicitly generated ideals or a Gr\"obner normal-form calculation. Their inputs and their mathematical contracts are therefore explicit and independently reproducible.

\section{Context, scope, and reproducibility}\label{sec:context}

Two essential ingredients predate the present construction. Green and
King already published the dimension-four, depth-two catalogue entry for
$\SG{128}{859}$~\cite{Green_2011}, and Okuyama already proved the
abstract associated-prime--centralizer theorem used in
Theorem~\ref{thm:okuyama}. These results identify a candidate and provide
the conceptual bridge, but neither one determines what happens for the
rank-two elementary abelian subgroups of this particular group. The
group-specific information needed for a disproof is the complete
rank-two subgroup enumeration, the associated centralizer
classification, and a depth-at-least-three certificate for every
centralizer type that occurs. Those finite steps are the load-bearing
new content of the argument.

The finite verification can be reconstructed as follows. Starting from the presentation in Appendix~\ref{app:group-presentation}, enumerate the group elements using a confluent polycyclic or coset-enumeration procedure. Identify all nonidentity elements of order two, test every commuting pair, retain the generated subgroups of order four, and deduplicate by complete element sets as in Section~\ref{sec:enumeration}. For each retained subgroup, compute its centralizer, Small Groups identifier, and center $2$-rank. The required output is Table~\ref{tab:centralizers}, including both the six isomorphism types and their multiplicities.

For the commutative-algebra part, form the three quotient rings from Appendices~\ref{app:ambient-ring} and~\ref{app:exceptional-rings} using exact arithmetic over $\F_2$. For $G$, compute the two successive ideal quotients in \eqref{eq:ambient-colons}. Equality with the current ideal verifies the two non-zero-divisor conditions. Next reduce $w$ and its products with all thirteen generators; $w$ must have nonzero normal form and every product must have zero normal form, recovering Table~\ref{tab:ambient-certificate}. For each exceptional centralizer, form the three successive ideal quotients associated with \eqref{eq:sequence-90} and \eqref{eq:sequence-216}; every quotient generator must reduce to zero modulo the ideal at that stage. Together, these calculations reproduce every finite premise used in the proof.

\section*{Statement on AI-assisted discovery and human verification}
The counterexample presented in this paper was initially identified by the TARS agent system through an autonomous mathematical search. The mathematical argument, part of the finite computations, and presentation certificates were subsequently examined and independently verified by X. Dai, who reorganized the content, supplemented logical chains and wrote the paper.

\appendix
\section{Finite group and cohomology-presentation data}\label{app:data}

This appendix records the finite and algebraic inputs used by the certificates so that the computational claims can be checked without reconstructing the presentations from external databases. Multiplication is written as \texttt{*} in the exact ASCII listings. Group relators are set equal to the identity, and cohomology-ring relations are set equal to zero. The long relation lists are data for exact verification; the conceptual role of each calculation is explained in Sections~\ref{sec:enumeration}, \ref{sec:ambient-depth}, and~\ref{sec:centralizers}.

\subsection{A finite presentation of \texorpdfstring{$G$}{G}}\label{app:group-presentation}

Let $F$ be the free group on $x_1,\ldots,x_7$. The group $G=\SG{128}{859}$, of order $128$ and structural type $(C_4\times C_4):(C_4\times C_2)$, is $F/N$, where $N$ is the normal closure of the following $28$ words.

\begin{Verbatim}[fontsize=\small,breaklines=true,breakanywhere=true]
x1^2*x5^-1,
x2^-1*x1^-1*x2*x1*x4^-1,
x3^-1*x1^-1*x3*x1,
x4^-1*x1^-1*x4*x1*x6^-1,
x5^-1*x1^-1*x5*x1,
x6^-1*x1^-1*x6*x1*x7^-1,
x7^-1*x1^-1*x7*x1,
x2^2,
x3^-1*x2^-1*x3*x2*x7^-1,
x4^-1*x2^-1*x4*x2*x7^-1,
x5^-1*x2^-1*x5*x2*x7^-1*x6^-1,
x6^-1*x2^-1*x6*x2,
x7^-1*x2^-1*x7*x2,
x3^2,
x4^-1*x3^-1*x4*x3,
x5^-1*x3^-1*x5*x3,
x6^-1*x3^-1*x6*x3,
x7^-1*x3^-1*x7*x3,
x4^2*x7^-1,
x5^-1*x4^-1*x5*x4*x7^-1,
x6^-1*x4^-1*x6*x4,
x7^-1*x4^-1*x7*x4,
x5^2,
x6^-1*x5^-1*x6*x5,
x7^-1*x5^-1*x7*x5,
x6^2,
x7^-1*x6^-1*x7*x6,
x7^2.
\end{Verbatim}

\subsection{\texorpdfstring{$\SG{128}{859}$}{SmallGroup(128,859)}: structural and ring data}\label{app:ambient-ring}

\begin{table}[ht]
\centering
\caption{Catalogue data for $\SG{128}{859}$.}
\begin{tabular}{@{}ll@{}}
\toprule
Catalogue field & Value \\
\midrule
Minimal group generators; exponent & $3;\ 4$ \\
Group type; $2$-rank; center $2$-rank & nonabelian; $4;\ 1$ \\
Ranks of maximal elementary abelian classes & $3,4,4$ \\
Krull dimension; depth; Duflot bound & $4;\ 2;\ 1$ \\
Minimal ring generators; largest degree & $13;\ 8$ \\
Minimal relations; largest degree & $44;\ 14$ \\
Benson completion degree & $14$ (perfect completion) \\
$a$-invariants & $-\infty,-\infty,-6,-4,-4$ \\
\bottomrule
\end{tabular}
\end{table}

The catalogue Poincar\'e series is
\[
  P_G(t)=-\frac{t^7-t^6-t^5+t^3-t^2-1}
  {(t+1)(t-1)^4(t^2+1)(t^4+1)}.
\]
For a compact exact listing, use the aliases in Table~\ref{tab:aliases}.

\begin{table}[ht]
\centering
\caption{Aliases for the generators of $H^*(G;\F_2)$.}
\label{tab:aliases}
\begin{tabular}{@{}clc@{\qquad}clc@{}}
\toprule
Alias & Source generator & Degree & Alias & Source generator & Degree \\
\midrule
$a$ & $a_{1,0}$ & 1 & $h$ & $b_{5,24}$ & 5 \\
$b$ & $b_{1,1}$ & 1 & $i$ & $b_{5,25}$ & 5 \\
$c$ & $b_{1,2}$ & 1 & $j$ & $b_{6,35}$ & 6 \\
$d$ & $b_{2,4}$ & 2 & $k$ & $b_{6,36}$ & 6 \\
$e$ & $b_{2,5}$ & 2 & $l$ & $b_{7,49}$ & 7 \\
$f$ & $b_{2,6}$ & 2 & $m$ & $c_{8,65}$ & 8 \\
$g$ & $b_{3,11}$ & 3 & & & \\
\bottomrule
\end{tabular}
\end{table}

Here $a$ is nilpotent and $m$ is the catalogue's Duflot regular element. With the aliases above, the complete source relation ideal is generated by the following $44$ homogeneous polynomials.

\begin{Verbatim}[fontsize=\scriptsize,breaklines=true,breakanywhere=true]
r01 = a^2;
r02 = a*b;
r03 = d*b+e*a;
r04 = e*b+d*b;
r05 = b*c^2+f*a;
r06 = e^2+d*e+e*a*c;
r07 = a*g+f*a*c;
r08 = b*g+f*b*c;
r09 = f*a*c^2+f^2*a;
r10 = g^2+f^2*c^2+e*c^4+d*f^2+f^2*a*c;
r11 = a*h+f^2*a*c+d*f*a*c;
r12 = b*h+f*b^3*c+f^2*b*c+f^2*b^2+d*f*a*c;
r13 = a*i;
r14 = c^2*i+f*h+f*c^2*g+f^2*g+f^2*c^3+f^2*b^2*c+f^3*b+e*c^5+e*f*c^3+e*f^2*c+d*f*g+d*f^2*c+f^3*a+d*f^2*a;
r15 = e*i+e*h+e*f^2*c+d*e*g+d*e*c^3+d*e*f*c+d*f^2*a+d^2*f*a;
r16 = e*h+e*f^2*c+d*i+d*e*g+d*e*c^3+d*e*f*c+d*f^2*a+d^2*f*a;
r17 = e*h+e*f*g+d*e*g+j*a+f^3*a+d^2*e*a;
r18 = b*c*i+j*b+f*b^4*c+f^2*b^2*c+f^2*b^3+e*h+e*f*g+d*e*g+f^3*a+d^2*e*a;
r19 = e*c^2*g+e*f*c^3+d*h+d*c^2*g+d*f*g+d*f*c^3+d*e*f*c+d^2*g+d^2*f*c+k*a+f^3*a+d*f^2*a+d^2*f*a;
r20 = b^2*i+k*b+f^3*b+e*h+e*f*g+d*e*g+d^2*f*a;
r21 = g*h+f*c*h+f^2*c*g+f^3*c^2+e*f*c^4+e*f^2*c^2+d*c*i+d*f^2*c^2+d*f^3+d*e*f*c^2+d^2*f^2+j*a*c+f^3*a*c+
      d^2*f*a*c+d^2*e*a*c;
r22 = g*i+f*c*i+e*j+e*f*c^4+e*f^2*c^2+e*f^3+d*c*i+d*c^3*g+d*j+d*f*c^4+d*f^2*c^2+d*e*c*g+d*e*c^4+
      d*e*f*c^2+d^2*c*g+d^2*f*a*c+d^2*e*a*c;
r23 = g*i+g*h+f*c*i+f*c*h+f^2*c*g+f^3*c^2+e*f^2*c^2+e*f^3+d*c*h+d*c^3*g+d*f*c*g+d*f*c^4+d*f^2*c^2+d*f^3+
      d*e*c^4+d*e*f*c^2+d^2*c*g+d^2*f*c^2+d^2*f^2+k*a*c+f^3*a*c+d^2*f*a*c;
r24 = g*i+g*h+f*c*i+f*c*h+f^2*c*g+f^3*c^2+e*k+e*j+e*f*c^4+e*f^2*c^2+d*f^2*c^2+d*f^3+d*e*f^2+d^2*f^2+
      d^2*e*c^2+d^3*e+d^2*f*a*c;
r25 = g*i+f*c*i+e*f*c^4+e*f^3+d*c*i+d*c*h+d*c^3*g+d*f*c*g+d*f*c^4+d*e*c^4+d^2*c*g+d^2*f*c^2+a*l+
      d*f^2*a*c;
r26 = b*l+f*b*i+f*b^5*c+f^2*b^3*c+f^3*b*c+f^3*a*c+d*f^2*a*c+d^2*f*a*c;
r27 = j*g+f*j*c+f^2*c^2*g+f^3*c^3+e*l+e*c^7+e*f*c^5+e*f^2*c^3+d*c^2*h+d*c^4*g+d*f*h+d*f*c^2*g+d*f*c^5+
      d*f^2*g+d*f^2*c^3+d*e*c^5+d*e*f*c^3+d^2*i+d^2*c^2*g+d^2*f*g+d^2*f*c^3+d^2*e*c^3+k*a*c^2+f*j*a+
      f^4*a+d*f^3*a+d^3*f*a+d^3*e*a;
r28 = k*g+f*c^4*g+f*k*c+f^2*c^2*g+f^2*c^5+f^3*g+f^3*c^3+f^4*c+d*l+d*c^2*h+d*c^4*g+d*k*c+d*j*c+d*f*c^2*g+
      d*f*c^5+d*f^2*g+d*f^2*c^3+d^2*i+d^2*h+d^2*f*g+d^2*f*c^3+d^2*f^2*c+d^2*e*f*c+d^3*e*c+f*j*a+f^4*a+
      d*k*a+d^2*f^2*a+d^3*f*a+d^3*e*a;
r29 = h^2+f^4*c^2+f^4*b^2+d*f^2*c^4+d*f^4+d*e*f^2*c^2+d^2*e*c^4+d^3*f^2+d*f^3*a*c+d^2*f^2*a*c;
r30 = h*i+f*j*b^2+f^2*c*i+f^2*b*i+f^2*b^5*c+f^3*b^3*c+f^3*b^4+e*f^2*c^4+e*f^4+d*f*c*h+d*f*c^3*g+
      d*f^2*c*g+d*f^2*c^4+d*e*f^3+d^2*c*i+d^2*c*h+d^2*c^3*g+d^2*f*c^4+d^2*f^2*c^2+d^2*e*c^4+d^3*c*g+
      d^3*f*c^2+d*k*a*c+d*j*a*c+d^2*f^2*a*c+d^3*e*a*c;
r31 = g*l+c^7*g+j*c^4+f*c*l+f*c^5*g+f*c^8+f^2*c^3*g+f^2*b*i+f^2*k+f^4*c^2+f^5+e*c*l+e*c^8+d*c*l+d*c^3*h+
      d*k*c^2+d*j*c^2+d*f*c^3*g+d*f*c^6+d*f^3*c^2+d*f^4+d*e*c^6+d*e*f*c^4+d*e*f^3+d^2*c*i+d^2*c^3*g+
      d^2*e*c*g+d^3*c*g+d^3*f*c^2+d^3*f^2+d^3*e*c^2+f*j*a*c+f^4*a*c+d*j*a*c+d*f^3*a*c+d^3*f*a*c+
      d^3*e*a*c;
r32 = i^2+k*b^4+f*k*b^2+f^2*b*i+f^2*b^5*c+f^2*b^6+f^3*b^4+f^4*b*c+e*f^2*c^4+e*f^4+d*e*c^6+d*e*f^2*c^2+
      f^4*a*c+m*b^2;
r33 = j*i+j*h+j*b^5+f*b^8*c+f^2*c^2*h+f^2*b^7+f^2*k*b+f^2*j*c+f^2*j*b+f^3*h+f^3*c^2*g+f^3*b^4*c+f^4*g+
      e*f^4*c+d*c^4*h+d*j*c^3+d*f*c^2*h+d*f*c^4*g+d*f^2*h+d*f^2*c^2*g+d*f^3*c^3+d*f^4*c+d*e*l+d*e*c^7+
      d*e*f*c^5+d*e*f^2*c^3+d^2*f*h+d^2*f*c^2*g+d^2*f*c^5+d^2*f^2*c^3+d^2*e*c^5+d^3*i+d^3*f*g+d^3*e*c^3+
      k*a*c^4+f^2*j*a+d^4*f*a+d^4*e*a+m*b^2*c;
r34 = k*i+k*b^5+j*h+f*k*b^3+f^2*c^4*g+f^2*b^6*c+f^2*b^7+f^2*k*b+f^2*j*c+f^2*j*b+f^3*i+f^3*h+f^3*c^5+
      f^3*b^4*c+f^3*b^5+f^4*g+e*f^2*c^5+e*f^4*c+d*c^6*g+d*j*c^3+d*f*c^2*h+d*f*c^7+d*f^2*i+d*f^2*h+
      d*f^4*c+d*e*l+d*e*f*c^5+d*e*f^2*c^3+d^2*c^4*g+d^2*f^2*g+d^2*e*f*c^3+d^2*e*f^2*c+d^3*f^2*c+
      d^3*e*c^3+d*f^4*a+d^4*f*a+d^4*e*a+m*b^3;
r35 = j*h+f^2*c^2*h+f^2*j*c+f^2*j*b+f^3*b^4*c+f^4*c^3+e*f*l+e*f*c^7+e*f^3*c^3+d*f*c^2*h+d*f*c^4*g+
      d*f*j*c+d*f^2*h+d*f^2*c^2*g+d*f^3*g+d*e*l+d*e*c^7+d^2*c^2*h+d^2*c^4*g+d^2*f*i+d^2*f*c^5+d^2*f^2*g+
      d^2*e*c^5+d^2*e*f^2*c+d^3*i+d^3*c^2*g+d^3*f*c^3+d^3*f^2*c+d^3*e*c^3+f^5*a+d*k*a*c^2+d*f*j*a+
      d*f^4*a+d^2*j*a+d^2*f^3*a+d^4*e*a+e*m*a;
r36 = k*h+f*c^4*h+f*j*b^3+f^2*c^2*h+f^2*b^6*c+f^2*k*c+f^2*k*b+f^3*h+f^3*c^5+f^3*b^4*c+f^3*b^5+f^4*c^3+
      f^5*c+f^5*b+e*c^2*l+e*f*c^7+e*f^2*c^5+d*c^2*l+d*c^4*h+d*c^6*g+d*k*c^3+d*j*c^3+d*f*l+d*f*c^2*h+
      d*f*c^4*g+d*f*c^7+d*f*k*c+d*f^2*h+d*f^2*c^2*g+d*f^2*c^5+d*f^3*c^3+d*e*c^7+d*e*f*c^5+d*e*f^2*c^3+
      d^2*l+d^2*c^2*h+d^2*k*c+d^2*j*c+d^2*f*i+d^2*f*c^2*g+d^2*f*c^5+d^2*e*c^5+d^2*e*f*c^3+d^3*i+d^3*f*g+
      d^3*f*c^3+d^3*e*c^3+d^4*g+d^4*f*c+d^4*e*c+k*a*c^4+f^5*a+d*k*a*c^2+d*f^4*a+d^2*k*a+d^2*j*a+
      d^2*f^3*a+d^3*f^2*a+d^4*f*a+d^4*e*a+d*m*a;
r37 = j*b^6+j*k+j^2+f*b^9*c+f*j*c^4+f^2*b^8+f^2*k*c^2+f^2*k*b^2+f^2*j*c^2+f^3*c^6+f^3*j+f^4*b^3*c+
      f^4*b^4+f^5*c^2+f^5*b^2+e*c^10+d*c^3*l+d*c^5*h+d*k*c^4+d*f*k*c^2+d*f*j*c^2+d*f^2*c^6+d*f^2*j+
      d*f^3*c^4+d*e*c*l+d*e*f*c^6+d*e*f^2*c^4+d^2*c*l+d^2*c^3*h+d^2*k*c^2+d^2*e*f*c^4+d^2*e*f^2*c^2+
      d^3*c*h+d^3*c^3*g+d^3*j+d^3*f*c*g+d^3*f*c^4+d^3*f^2*c^2+d^3*e*f*c^2+d^4*c*g+d^4*e*c^2+f^5*a*c+
      d*f*j*a*c+d*f^4*a*c+d^2*k*a*c+d^2*f^3*a*c+d^3*f^2*a*c+d^4*f*a*c+m*b^3*c;
r38 = j^2+f^4*c^4+f^4*b^4+e*c^10+e*f^3*c^4+e*f^4*c^2+d*f*c^5*g+d*f*j*c^2+d*f^2*c*h+d*f^2*c^3*g+d*f^3*c*g+
      d*e*c*l+d*e*f^2*c^4+d*e*f^3*c^2+d^2*c^3*h+d^2*j*c^2+d^2*f^2*c*g+d^2*f^3*c^2+d^2*e*f^2*c^2+d^3*c*h+
      d^3*c^3*g+d^3*f*c*g+d^3*f^2*c^2+d^3*e*c*g+d^3*e*c^4+d^4*c*g+d^4*f*c^2+d^4*e*c^2+f^2*j*a*c+
      d*k*a*c^3+d*f*j*a*c+d*f^4*a*c+d^2*k*a*c+d^2*j*a*c+d^2*f^3*a*c+d^4*f*a*c+d^4*e*a*c+d*e*m+e*m*a*c;
r39 = h*l+f*c^5*h+f*c^7*g+f*j*c^4+f^2*c*l+f^2*c^3*h+f^2*c^5*g+f^2*c^8+f^2*k*c^2+f^2*j*c^2+f^2*j*b^2+
      f^3*k+f^4*b^4+f^5*b*c+f^6+e*c^3*l+d*c^3*l+d*c^5*h+d*k*c^4+d*f*c*l+d*f*c^3*h+d*f*c^5*g+d*f*k*c^2+
      d*f*j*c^2+d*f^2*c*h+d*f^2*c^3*g+d*f^2*k+d*f^3*c*g+d*f^3*c^4+d*f^4*c^2+d*e*c*l+d*e*f*c^6+
      d*e*f^2*c^4+d*e*f^3*c^2+d*e*f^4+d^2*c*l+d^2*c^3*h+d^2*c^5*g+d^2*k*c^2+d^2*j*c^2+d^2*f*c*i+
      d^2*f*c^6+d^2*f^2*c*g+d^2*f^4+d^2*e*f*c^4+d^2*e*f^2*c^2+d^2*e*f^3+d^3*c*h+d^3*c^3*g+d^3*f^3+
      d^3*e*c*g+d^3*e*f*c^2+d^4*f^2+d^4*e*c^2+k*a*c^5+f^2*j*a*c+f^5*a*c+d*k*a*c^3+d*f^4*a*c+d^2*j*a*c+
      d^3*f^2*a*c+d^4*f*a*c+d^4*e*a*c+e*m*a*c+d*m*a*c;
r40 = k*b^6+k^2+f*k*b^4+f^2*c^8+f^2*b^7*c+f^2*b^8+f^2*k*b^2+f^3*b^6+f^4*c^4+f^4*b^3*c+f^5*b^2+f^6+
      e*f*c^8+e*f^2*c^6+d*c^5*h+d*c^7*g+d*k*c^4+d*j*c^4+d*f*c^5*g+d*f*c^8+d*f*k*c^2+d*f^2*c^3*g+
      d*e*f*c^6+d*e*f^3*c^2+d*e*f^4+d^2*c*l+d^2*c^5*g+d^2*j*c^2+d^2*f*k+d^2*f*j+d^2*f^4+d^2*e*c^6+
      d^2*e*f^3+d^3*c^3*g+d^3*k+d^3*j+d^3*f^3+d^3*e*c*g+d^3*e*f*c^2+d^4*f^2+d^5*f+k*a*c^5+f^5*a*c+
      d*f*j*a*c+d*f^4*a*c+d^2*k*a*c+d^3*f^2*a*c+d^4*f*a*c+m*b^4+d^2*m;
r41 = i*l+f*c^7*g+f*k*b^4+f*j*c^4+f*j*b^4+f^2*c^3*h+f^2*c^5*g+f^2*c^8+f^2*b^7*c+f^2*k*b^2+f^2*j*b^2+
      f^3*c*h+f^3*c^3*g+f^3*b*i+f^3*b^5*c+f^3*j+f^4*c*g+f^5*c^2+f^5*b*c+f^5*b^2+e*c^3*l+e*f*c*l+e*f*c^8+
      e*f^2*c^6+d*f*c^5*g+d*f^2*c*i+d*f^2*c^3*g+d*f^3*c^4+d*f^4*c^2+d*e*f^3*c^2+d*e*f^4+d^2*f*c*i+
      d^2*e*f*c^4+d^2*e*f^2*c^2+f^5*a*c+f*m*b^2;
r42 = c^10*g+j*l+j*c^7+f*c^11+f*j*c^5+f*j*b^5+f^2*c^2*l+f^2*c^9+f^2*b^8*c+f^2*k*c^3+f^2*j*c^3+f^3*c^2*h+
      f^3*c^4*g+f^3*b^6*c+f^3*b^7+f^3*k*b+f^4*h+f^4*c^5+f^4*b^4*c+f^5*g+f^5*b^2*c+e*c^11+e*f*c^9+e*f^5*c+
      d*c^4*l+d*c^8*g+d*k*c^5+d*j*c^5+d*f*c^2*l+d*f^2*c^7+d*f^2*k*c+d*f^3*c^2*g+d*f^3*c^5+d*f^4*g+
      d*f^4*c^3+d*e*f*l+d*e*f^2*c^5+d*e*f^3*c^3+d^2*c^2*l+d^2*c^4*h+d^2*c^6*g+d^2*k*c^3+d^2*j*c^3+
      d^2*f*c^2*h+d^2*f*c^4*g+d^2*f*j*c+d^2*f^2*h+d^2*f^3*g+d^2*f^3*c^3+d^2*f^4*c+d^2*e*c^7+d^2*e*f*c^5+
      d^2*e*f^2*c^3+d^3*f*i+d^3*f*h+d^3*f*c^2*g+d^3*f*c^5+d^3*f^2*c^3+d^3*f^3*c+d^3*e*c^5+d^4*c^2*g+
      d^4*f*g+d^4*f*c^3+d^4*e*c^3+k*a*c^6+d*f^5*a+d^3*f^3*a+d^4*f^2*a+d^5*f*a+f*m*b^2*c+e*m*g+e*f*m*c+
      d*e*m*a;
r43 = k*l+f*c^4*l+f*c^6*h+f*k*c^5+f*k*b^5+f*j*b^5+f^2*c^2*l+f^2*c^4*h+f^2*c^6*g+f^2*c^9+f^2*b^8*c+
      f^2*k*c^3+f^2*k*b^3+f^2*j*b^3+f^3*l+f^3*b^6*c+f^3*k*b+f^3*j*b+f^4*c^2*g+f^4*c^5+f^5*b^3+f^6*b+
      e*f*c^9+e*f^4*c^3+d*k*c^5+d*j*c^5+d*f*c^2*l+d*f*c^6*g+d*f^2*l+d*f^2*c^2*h+d*f^2*c^4*g+d*f^2*k*c+
      d*f^3*i+d*f^3*h+d*f^3*c^2*g+d*e*c^2*l+d*e*f*c^7+d*e*f^2*c^5+d*e*f^3*c^3+d*e*f^4*c+d^2*c^2*l+
      d^2*c^4*h+d^2*f*l+d^2*f*c^7+d^2*f*j*c+d^2*f^2*i+d^2*f^2*c^5+d^2*e*l+d^3*c^4*g+d^3*k*c+d^3*j*c+
      d^3*f*i+d^3*f*h+d^3*f*c^5+d^3*f^2*g+d^3*f^3*c+d^3*e*c^5+d^3*e*f^2*c+d^4*h+d^4*c^2*g+d^4*f*c^3+
      d^4*e*c^3+d^5*g+d^5*e*c+k*a*c^6+d*f^2*j*a+d*f^5*a+d^3*k*a+d^3*f^3*a+d^5*f*a+f*m*b^3+d*m*g+d*f*m*c+
      d*e*m*c+d^2*m*c+d*m*a*c^2+d*e*m*a;
r44 = l^2+f*c^9*g+f*j*c^6+f^2*k*c^4+f^2*k*b^4+f^2*j*c^4+f^3*c^8+f^3*k*c^2+f^3*k*b^2+f^4*c^3*g+f^4*c^6+
      f^4*b*i+f^4*b^5*c+f^4*b^6+f^5*b^4+f^6*b*c+e*c^5*l+e*c^12+e*f*c^10+e*f^3*c^6+e*f^4*c^4+e*f^5*c^2+
      e*f^6+d*c^9*g+d*k*c^6+d*f*c^10+d*f*k*c^4+d*f*j*c^4+d*f^2*c*l+d*f^2*c^3*h+d*f^2*c^8+d*f^2*j*c^2+
      d*f^3*c^3*g+d*f^3*c^6+d*f^3*k+d*f^3*j+d*e*c^3*l+d*e*f*c^8+d*e*f^2*c^6+d*e*f^4*c^2+d*e*f^5+
      d^2*c^3*l+d^2*f*c^5*g+d^2*f*k*c^2+d^2*f*j*c^2+d^2*f^2*c^6+d^2*f^2*k+d^2*f^2*j+d^2*f^3*c^4+d^2*f^5+
      d^2*e*f*c^6+d^2*e*f^2*c^4+d^2*e*f^3*c^2+d^2*e*f^4+d^3*k*c^2+d^3*j*c^2+d^3*f*c*i+d^3*f^2*c^4+
      d^3*f^3*c^2+d^3*f^4+d^3*e*c^6+d^3*e*f*c^4+d^4*f^2*c^2+d^4*f^3+d^4*e*c^4+d^5*f*c^2+d^5*f^2+
      d^5*e*c^2+f^3*j*a*c+d^2*k*a*c^3+d^5*f*a*c+f^2*m*b^2+e*m*c^4+d*f^2*m+d*e*m*c^2+d^2*m*c^2+f^2*m*a*c+
      d*f*m*a*c;
\end{Verbatim}

\subsection{Exceptional centralizer presentations}\label{app:exceptional-rings}

For both exceptional centralizers, the catalogue data give Krull dimension $4$, depth $3$, Duflot bound $2$, $a$-invariants $-\infty,-\infty,-\infty,-4,-4$, and perfect Benson completion in degree $6$. The remaining structural data are shown in Table~\ref{tab:exceptional-data}.

\begin{table}[ht]
\centering
\caption{Catalogue data for the two exceptional centralizers.}
\label{tab:exceptional-data}
\begin{tabular}{@{}lcc@{}}
\toprule
Field & $\SG{64}{90}$ & $\SG{64}{216}$ \\
\midrule
Minimal group generators & 3 & 4 \\
Exponent; $2$-rank; center $2$-rank & $4;4;2$ & $4;4;2$ \\
Ranks of maximal elementary abelian classes & $4,4$ & $3,3,4,4$ \\
Ring generators / relations & $9/14$ & $7/5$ \\
Largest generator / relation degree & $4/6$ & $4/6$ \\
Poincar\'e-series denominator & $(t+1)(t-1)^4$ & $(t-1)^4(t^2+1)$ \\
\bottomrule
\end{tabular}
\end{table}
The numerator of each displayed Poincar\'e series is $1$.

\subsubsection{\texorpdfstring{$\SG{64}{90}$}{SmallGroup(64,90)}}

Write
\[
  S_{90}=\F_2[a,b,c,d,e,f,g,h,q]
\]
with weights
\[
  |a|=|b|=|c|=1,\qquad
  |d|=|e|=|f|=2,\qquad
  |g|=|h|=3,\qquad |q|=4,
\]
corresponding respectively to
\[
  a_{1,0},b_{1,1},c_{1,2},b_{2,4},b_{2,5},b_{2,6},b_{3,11},b_{3,12},c_{4,21}.
\]
The relation ideal is generated by
\begin{align*}
&a^2,\ ab,\ db+ea,\ eb+db,\ db+fa,\ e^2+df,\\
&ag,\ bg,\ ah,\ fg+eh+dea,\ eg+dh+dea,\\
&g^2+df^2+d^2f,\quad gh+ef^2+def,\\
&h^2+fbh+f^3+df^2+qb^2.
\end{align*}
The regular sequence \eqref{eq:sequence-90} becomes
\[
  c,\qquad q,\qquad b^2+f+e+d.
\]

\subsubsection{\texorpdfstring{$\SG{64}{216}$}{SmallGroup(64,216)}}

Write
\[
  S_{216}=\F_2[x,y,z,w,u,v,q]
\]
with weights
\[
  |x|=|y|=|z|=|w|=1,\qquad |u|=2,\qquad |v|=3,\qquad |q|=4,
\]
corresponding respectively to
\[
  b_{1,0},b_{1,1},b_{1,2},b_{1,3},c_{2,8},b_{3,15},c_{4,25}.
\]
The relation ideal is generated by
\begin{align*}
  &xz,\\
  &yz+y^2+xw+xy,\\
  &xw^2,\\
  &xv+uxw,\\
  &v^2+zw^2v+qz^2+uw^4+uzw^3+uz^2w^2+uy^3w+ux^2yw+u^2w^2.
\end{align*}
The regular sequence \eqref{eq:sequence-216} becomes
\[
  u,\qquad q,\qquad w^2+zw+z^2+x^2.
\]

\bibliographystyle{plain}
\begingroup
\catcode`<=\active
\def<i>#1</i>{\ \emph{#1}}
\bibliography{
  bibtex_month_macros,
  bib_exports/carlson_1995,
  bib_exports/duflot_1981,
  bib_exports/green_2003,
  bib_exports/green_king_2011,
  bib_exports/okuyama_2010,
  bib_exports/schafer_2019,
  bib_exports/garaialde_et_al_2022,
  bib_exports/schafer_2025
}
\endgroup

\end{document}